\documentclass[11pt]{amsart}
\usepackage{amsmath,amsfonts,amssymb,amsthm, a4wide, url, mathscinet,mathdots}

\newtheorem{theorem}{Theorem}[section]
\newtheorem{lemma}[theorem]{Lemma}
\newtheorem{claim}[theorem]{Claim}
\newtheorem*{definition*}{Defenition}
\newtheorem{proposition}[theorem]{Proposition}
\newtheorem{problema}{Problem}
\newtheorem{thm}{Theorem}

\newcommand{\prdct}[2]{W_{#2}\left(#1\right)}

\begin{document}
\title{The Advantage of Truncated Permutations}
\author{Shoni Gilboa}
\address{The Open University of Israel, Raanana 4353701, Israel.}
\author{Shay Gueron}
\address{University of Haifa, Haifa 3498838, Israel, and Amazon, USA. }

\begin{abstract}
Constructing a Pseudo Random Function (PRF) is a fundamental problem in cryptology. 
Such a construction, implemented by truncating the last $m$ bits of permutations of $\{0, 1\}^{n}$ was suggested by Hall et al. (1998). They conjectured that the distinguishing advantage of an adversary with $q$ queries, ${\bf Adv}_{n, m} (q)$, is small if $q = o (2^{(n+m)/2})$, established an upper bound on ${\bf Adv}_{n, m} (q)$ that confirms the conjecture for $m < n/7$, and also declared a general lower bound ${\bf Adv}_{n,m}(q)=\Omega(q^2/2^{n+m})$. The conjecture was essentially confirmed by Bellare and Impagliazzo (1999). Nevertheless, the problem of {\em estimating} ${\bf Adv}_{n, m} (q)$ remained open. Combining the trivial bound $1$, the birthday bound, and a result of Stam (1978) leads to the upper bound
\begin{equation*}
{\bf Adv}_{n,m}(q) = O\left(\min\left\{\frac{q(q-1)}{2^n},\,\frac{q}{2^{\frac{n+m}{2}}},\,1\right\}\right).
\end{equation*}
In this paper we show that this upper bound is tight for every $0\leq m<n$ and any $q$. 
This, in turn, verifies that the converse to the conjecture of Hall et al. is also correct, i.e., that ${\bf Adv}_{n, m} (q)$ is negligible only for $q = o (2^{(n+m)/2})$. 
\end{abstract}


\maketitle

\section{Introduction}
\label{sec:intro}
For every positive integer $k$, denote ${\mathcal B}_k:=\{0,1\}^k$. 
For positive integers $\ell, n$, let ${\mathcal F }_{n,\ell}$ be the set of functions from ${\mathcal B}_{n}$ to ${\mathcal B}_{\ell}$. 
A Pseudo Random Function (PRF) from ${\mathcal B}_{n}$ to ${\mathcal B}_{\ell}$ is a random variable taking values in ${\mathcal F}_{n,\ell}$.
The quality of a PRF $\Phi$ is determined by the ability of an ``adversary" to distinguish an instance of $\Phi$ from a function chosen uniformly at random from ${\mathcal F }_{n,\ell}$, in the following setting. It is assumed that the adversary has only query access to a function $\varphi: {\mathcal B}_{n} \to {\mathcal B}_{\ell}$, which is either selected uniformly at random from ${\mathcal F }_{n,\ell}$, or is an instance of the PRF $\Phi$. 
The adversary may use any algorithm ${\mathcal A}$ that first selects (possibly adaptively) a sequence of queries to the function, i.e., strings in ${\mathcal B}_n$, and then outputs a bit that we may interpret as the guess of ${\mathcal A}$. For $b \in \{0, 1\}$, let $P_{\Phi}^{{\mathcal A}} (b)$ be the probability that the output is $b$ when $\varphi$ is the PRF, and let $P_{U}^{{\mathcal A}} (b)$ be the probability that the output is $b$ when $\varphi$ is selected from ${\mathcal F }_{n,\ell}$ uniformly at random. 
The {\em advantage} of the algorithm ${\mathcal A}$ against the PRF $\Phi$ is defined as
$ \left| P_{\Phi}^{{\mathcal A}} (1) - P_{U}^{{\mathcal A}} (1) \right|$ (which also equals 
$\left| P_{\Phi}^{{\mathcal A}} (0) - P_{U}^{{\mathcal A}} (0) \right|$). 
The advantage of the adversary against the PRF $\Phi$ is the maximal advantage of ${\mathcal A}$ against $\Phi$ over all the algorithms it may use, as a function of the number of queries.
Hereafter, we consider adversaries with no computational limitations.

The classical example of a PRF from ${\mathcal B}_{n}$ to ${\mathcal B}_{n}$ is a permutation of ${\mathcal B}_{n}$ chosen uniformly at random. 
The advantage {\bf Adv} of this PRF is given by
\begin{align}
{\bf Adv} (q)&=1-\left(1-\frac{1}{2^n}\right)\left(1-\frac{2}{2^n}\right)\cdots\left(1-\frac{\min\{q,2^n\}-1}{2^n}\right)\nonumber\\
&=\Theta\left(\min\left\{\frac{q(q-1)}{2^n},\,1\right\}\right), \label{eq:birthday}
\end{align}
achieved by an adversary that executes the ``collision test'' (i.e., submits $\min\{q,2^n\}$ distinct queries and outputs $1$ if no two replies are equal, and $0$ otherwise). 
This implies that the number of queries required to distinguish a random permutation from a random function, with success probability significantly larger than, say, $1/2$, is $\Theta(2^{n/2})$. In other words, a permutation can be used safely (e.g., as a one-time-pad) as long as the number of outputs ($q$) that it produces is sufficiently lower than $2^{n/2}$.

\smallskip
A generalization of the above PRF is the following. 
\begin{definition*}
For integers $0\leq m<n$, let $\texttt{TRUNC}_{n, m}\in{\mathcal F}_{n,n-m}$ be defined by
$(x_1, x_2, \ldots x_{n})$ $\mapsto$ $(x_1, x_2, \ldots x_{n-m})$. The 
``Truncated Permutation" PRF from ${\mathcal B}_n$ to ${\mathcal B}_{n-m}$ is the composition
$\texttt{TRUNC}_{n, m} \circ \pi$, where $\pi$ is a permutation of ${\mathcal B}_{n}$ chosen uniformly at random. 
We denote the advantage of an (computationally unbounded) adversary against this PRF by ${\bf Adv}_{n, m}$.
\end{definition*}
Clearly, ${\bf Adv}_{n,m}(q)={\bf Adv}_{n,m}(\min\{q,2^n\})$, so we may restrict our attention to $q\leq 2^n$.

\medskip
The following problem arises naturally.
\begin{problema}\label{problem1}
For every $0\leq m<n$ and $q\leq 2^n$, find (the order of magnitude of) ${\bf Adv}_{n,m}(q)$.
\end{problema} 
A different, related, problem is the following.
\begin{problema}\label{problem2}
For every $0\leq m<n$, how many queries does the adversary need in order to gain non-negligible advantage against the Truncated Permutation PRF?
Specifically, what is (the order of magnitude) of $q_{1/2}(n,m)=\min\{q\mid {\bf Adv}_{n,m}(q)\geq 1/2\}$?
\end{problema}
Note that the classical `birthday bounds' 
\begin{equation}\label{eq:birthday_upper}
{\bf Adv}_{n,m}(q)\leq1-\left(1-\frac{1}{2^n}\right)\left(1-\frac{2}{2^n}\right)\cdots\left(1-\frac{q-1}{2^n}\right)\leq\min\left\{\frac{q(q-1)}{2^{n+1}},\,1\right\},
\end{equation} 
and hence $q_{1/2}(n,m)=\Omega(2^{n/2})$, are obviously valid. 
Indeed, every algorithm that the adversary can use with the truncated replies of $(n-m)$ bits from $\pi(w)$ ($w \in {\mathcal B}_n$) can also be used by the adversary who sees the full $\pi(w)$ (it can simply ignore $m$ bits and apply the same algorithm).
Of course, we expect `better' bounds that would reflect the fact that the adversary receives less information when $\pi(w)$ is truncated, and would allow for using the outputs of a (truncated) permutation for significantly more than $2^{n/2}$ times. 

Problems \ref{problem1} and \ref{problem2} were studied by Hall et al. \cite{Hall} in 1998, where the truncated (random) permutation were proposed as a PRF construction. 
They declared\footnote{The paper \cite{Hall} only provide a sketch of proof of \eqref{eq:Hall_lower} and claims that the computation may be completed by using techniques presented in the paper. We could not see how this is the case.
We therefore refer to \eqref{eq:Hall_lower} only as a `declared' result.} the lower bound 
\begin{equation}\label{eq:Hall_lower}
{\bf Adv}_{n,m}(q)=\Omega(q^2/2^{n+m})
\end{equation}
for every $0\leq m<n$ and $q\leq 2^{(n+m)/2}$. This bound implies that $q_{1/2}(n,m)=O (2^{(n+m)/2})$ for every $0\leq m<n$. 
Hall et al. also proved in \cite{Hall} the upper bound
\begin{equation}\label{Hall_result}
{\bf Adv}_{n,m}(q) \le 5\left(\frac{q}{2^{\frac{n+m}{2}}}\right)^{\frac{2}{3}}+\frac{1}{2}\left(\frac{q}{2^{\frac{n+m}{2}}}\right)^3\frac{1}{ 2^{\frac{n-7m}{2}}}\end{equation}
for every $0\leq m<n$ and any $q$. For $m\leq n/7$ this implies that $q_{1/2}(n,m)=\Omega(2^{(n+m)/2})$. However, for larger values of $m$, the bound on $q_{1/2}(n,m)$ that is offered by \eqref{Hall_result} deteriorates, and becomes (already for $m>n/4$) worse than the trivial birthday bound $q_{1/2}(n,m)=\Omega(2^{n/2})$. 
They conjectured that an adversary needs $\Omega(2^{(n+m)/2})$ queries in order to get non-negligible advantage, in the general case. 

It was shown in \cite[Theorem 4.2]{BI} that 
\begin{equation}\label{eq:BI}{\bf Adv}_{n,m}(q)=O(n)\frac{q}{2^{\frac{n+m}{2}}}\end{equation}
whenever $2^{n-m}<q<2^{\frac{n+m}{2}}$. 
This implies that $q_{1/2}=\Omega(\frac{1}{n}2^{\frac{n+m}{2}})$ for $m>\frac{1}{3}n+\frac{2}{3}\log_2 n+\Omega(1)$.

The method used to show \eqref{Hall_result} can be pushed to prove the conjecture made in \cite{Hall}, thus settling Problem~\ref{problem2}, for almost every $m$. In particular, it was shown in \cite{GG} that 
\begin{equation}\label{eq:GG1}
{\bf Adv}_{n,m}(q)\leq 2\sqrt[3]{2}\left(\frac{q}{2^{\frac{n+m}{2}}}\right)^{\frac{2}{3}}+\frac{2\sqrt{2}}{\sqrt{3}}\left(\frac{q}{2^{\frac{n+m}{2}}}\right)^{\frac{3}{2}}+\left(\frac{q}{2^{\frac{n+m}{2}}}\right)^2\end{equation}
for $ m\leq \frac{n}{3}$ and that
\begin{equation}\label{eq:GG2}
{\bf Adv}_{n,m}(q)\leq 3\left(\frac{q}{2^{\frac{n+m}{2}}}\right)^{\frac{2}{3}}+2\left(\frac{q}{2^{\frac{n+m}{2}}}\right)+5\left(\frac{q}{2^{\frac{n+m}{2}}}\right)^2+ 
\frac{1}{2}\left(\frac{2q}{2^{\frac{n+m}{2}}}\right)^{\frac{n}{n-m}}\end{equation}
for $\frac{n}{3}<m \leq n-\log_2(16n)$.
This implies that $q_{1/2}(n,m)=\Omega(2^{\frac{m+n}{2}})$ for every $0 \leq m \leq n-\log_2 (16n)$.

Surprisingly, it turns out that Problem~\ref{problem2} was solved, in a different context, $20$ years before it was raised in \cite{Hall}. The bound 
\begin{equation}\label{eq:Stam}
{\bf Adv}_{n,m}(q)\leq\frac{1}{2}\sqrt{\frac{(2^{n-m}-1)q(q-1)}{(2^n-1)(2^n-(q-1))}}
\leq\frac{1}{2\sqrt{1-\frac{q-1}{2^n}}}\cdot\frac{q}{2^{\frac{n+m}{2}}},\end{equation}
which is valid for every $0\leq m<n$ and $q\leq 2^n$, follows directly from a result of Stam \cite[Theorem 2.3]{Stam}. This implies that 
$q_{1/2}(n,m)=\Omega(2^{(n+m)/2})$ for every $0 \leq m<n$, confirming the conjecture of \cite{Hall} in all generality.

\medskip
This settles Problem~\ref{problem2}, but Problem~\ref{problem1} still remains quite open. 
Note that the bound \eqref{eq:Stam} is tighter than the bounds \eqref{Hall_result}, \eqref{eq:BI}, \eqref{eq:GG1} and \eqref{eq:GG2}.
Therefore, summarizing the above results, the best known upper bound for the advantage in Problem~\ref{problem1}, is the one obtained by combining \eqref{eq:birthday_upper} and \eqref{eq:Stam}, namely 
 \begin{align}
{\bf Adv}_{n,m}(q)&\leq\min\left\{\frac{q(q-1)}{2^{n+1}},\,\frac{1}{2}\sqrt{\frac{(2^{n-m}-1)q(q-1)}{(2^n-1)(2^n-(q-1))}},\, 1\right\}\label{eq:upper}\\
&=\Theta\left(\min\left\{\frac{q(q-1)}{2^n},\,\frac{q}{2^{\frac{n+m}{2}}},\,1\right\}\right),\nonumber
\end{align}
whereas the only general lower bound that we are aware of is the bound \eqref{eq:Hall_lower}, declared in \cite{Hall}.
It follows from \eqref{eq:birthday} that the bound \eqref{eq:upper} is tight if $m=0$, and it was shown in \cite{GGM} that it is tight also in the case $m=n-1$. 

In this paper we settle Problem \ref{problem1} by showing that \eqref{eq:upper} is always tight, as formulated in the following theorem.
\begin{thm}\label{thm:main}
For every $0\leq m<n$ and any $q$,
\begin{equation*}
{\bf Adv}_{n,m}(q)=\Theta\left(\min\left\{\frac{q(q-1)}{2^n},\,\frac{q}{2^{\frac{n+m}{2}}},\,1\right\}\right).
\end{equation*}
\end{thm}
In particular, note that this implies that the bound \eqref{eq:Hall_lower} is, in general, not tight. 

\smallskip
We point out that the proof of Theorem \ref{thm:main} shows that the lower bound still holds if the adversary can use only computatinally efficient algorithms.

\medskip
A short version of this paper, with only a hint of the proof, appears in \cite{GG_CSCML}.

\section{Notation and Preliminaries}
\label{sec:notation}
For $0\leq m<n$ and $1\leq q\leq 2^n$, we view $({\mathcal B}_{n-m})^q$ as the set of all possible sequences of replies that the adversary gets for his $q$ queries. 
We remark here that in our problem, we may assume that all the queries are fixed and distinct.
For every $\omega=(\omega_i)_{i=1}^q\in({\mathcal B}_{n-m})^q$ and $\alpha\in{\mathcal B}_{n-m}$ let 
\begin{equation*}
d_{\alpha}(\omega):=\#\{1\leq i\leq q\mid\omega_i=\alpha\},
\end{equation*}
i.e., $d_{\alpha}(\omega)$ is the number of times $\alpha$ appears in the sequence $\omega$.
For every positive real $t$, let $\prdct 0 t:=1$ and for every positive integer $k$, 
\begin{equation*}
\prdct k t:= \prod_{j=0}^{k-1}\left(1-\frac{j}{t}\right).
\end{equation*}
As in Section \ref{sec:intro}, consider an adversary that has only query access to a function $\varphi: {\mathcal B}_{n} \to {\mathcal B}_{n-m}$, which is either selected uniformly at random from ${\mathcal F }_{n,n-m}$, or is $\texttt{TRUNC}_{n, m} \circ \pi$, where $\pi$ is a permutation of ${\mathcal B}_{n}$ chosen uniformly at random. 
For every $\omega\in({{\mathcal B}}_{n-m})^q$, the probability that $\omega$ is the actual sequence of replies that the adversary gets for his queries is obviously $\frac{1}{2^{(n-m)q}}$ in the former case, and it is easy to verify that it is $\frac{1}{2^{(n-m)q}}R(\omega)$ in the latter, where
\begin{equation*}
R(\omega):=\frac{\prod_{\alpha\in{\mathcal B}_{n-m}}\prdct{d_{\alpha}(\omega)}{2^m}}{\prdct q {2^n}}.
\end{equation*} 
Suppose that the adversary uses an algorithm ${\mathcal A}$ and let $S_{\mathcal A}\subseteq({\mathcal B}_{n-m})^q$ be the set of sequences of replies for which ${\mathcal A}$ outputs $1$. Then,
\begin{equation*}
P_{U}^{{\mathcal A}}(1)=\frac{1}{2^{(n-m)q}}|S_{\mathcal A}|,\quad
P_{\texttt{TRUNC}_{n, m} \circ \pi}^{{\mathcal A}} (1)=\frac{1}{2^{(n-m)q}}\sum_{\omega\in S_{\mathcal A}}R(\omega),
\end{equation*}
and the advantage of ${\mathcal A}$ against the PRF $\texttt{TRUNC}_{n, m} \circ \pi$ is therefore
$\frac{1}{2^{(n-m)q}}|\sum_{\omega\in S_{\mathcal A}}\left(R(\omega)-1\right)|$.

We conclude that
\begin{equation}\label{eq:AdvR}
{\bf Adv}_{n,m}(q)=\max_{S\subseteq{({\mathcal B}_{n-m})^q}}\frac{1}{2^{(n-m)q}}\left\lvert\sum_{\omega\in S}\left(R(\omega)-1\right)\right\rvert.
\end{equation}

\section{Proof of Theorem \ref{thm:main}}
We first address the regime $q\leq 2^{\frac{n-m}{2}+8}$, in which
\begin{equation*}
\min\left\{\frac{q(q-1)}{2^n},\,\frac{q}{2^{\frac{n+m}{2}}},\,1\right\}=\Theta\left(\frac{q(q-1)}{2^n}\right).
\end{equation*}
\begin{proposition}\label{prop:small}
If $q\leq 2^{\frac{n-m}{2}+8}$, then
\begin{equation*}
{\bf Adv}_{n,m}(q)=\Omega\left(\frac{q(q-1)}{2^n}\right).
\end{equation*}
\end{proposition}
\begin{proof} 
Assume first, in addition, that $q\leq 2^{n-m-1}$. Let 
\begin{equation*}
S:=\left\{\omega\in{({\mathcal B}_{n-m})^q}\mid \forall\alpha\in{\mathcal B}_{n-m}:d_{\alpha}(\omega)\leq 1\right\}.
\end{equation*}
For every $\omega\in S$,
\begin{equation*}
R(\omega)=\frac{1}{\prdct q {2^n}}=\prod_{j=0}^{q-1}\frac{1}{1-\frac{j}{2^n}}\geq \prod_{j=0}^{q-1}\left(1+\frac{j}{2^n}\right)
\end{equation*}
and hence
\begin{equation*}
R(\omega)-1\geq\sum_{j=0}^{q-1}\frac{j}{2^n}=\frac{q(q-1)/2}{2^n}.
\end{equation*}
For every $1\leq k\leq q-1$ we have
\begin{equation*}
\left(1-\frac{k}{2^{n-m}}\right)\left(1-\frac{q-k}{2^{n-m}}\right)\geq 1-\frac{q}{2^{n-m}},
\end{equation*}
and hence, by Bernoulli's inequality,
\begin{align*}
\frac{|S|}{2^{(n-m)q}}&=\left(1-\frac{1}{2^{n-m}}\right)\left(1-\frac{2}{2^{n-m}}\right)\cdots\left(1-\frac{q-1}{2^{n-m}}\right)\geq\left(1-\frac{q}{2^{n-m}}\right)^{\frac{q-1}{2}}\\
&=\left(\left(1-\frac{q}{2^{n-m}}\right)^{\frac{q-1}{2^{17}}}\right)^{2^{16}}\geq\left(1-\frac{q-1}{2^{17}}\cdot\frac{q}{2^{n-m}}\right)^{2^{16}}>\left(1-\frac{q^2}{2^{n-m+17}}\right)^{2^{16}}\geq\left(\frac{1}{2}\right)^{2^{16}}.
\end{align*}
Therefore, by \eqref{eq:AdvR},
\begin{gather*}{\bf Adv}_{n,m}(q)\geq\frac{1}{2^{(n-m)q}}\left|\sum_{\omega\in S}\left(R(\omega)-1\right)\right|\geq\frac{|S|}{2^{(n-m)q}}\cdot\frac{q(q-1)/2}{2^n}\geq\left(\frac{1}{2}\right)^{2^{16}+1}\frac{q(q-1)}{2^n}.
\end{gather*}
Finally, if $2^{n-m-1}+1\leq q\leq 2^{\frac{n-m}{2}+8}$, then by what we already proved, 
\begin{align*}{\bf Adv}_{n,m}(q)&\geq {\bf Adv}_{n,m}(2^{n-m-1}+1)\geq\left(\frac{1}{2}\right)^{2^{16}+1}\frac{\left(2^{n-m-1}+1\right)2^{n-m-1}}{2^n}\\
&\geq\left(\frac{1}{2}\right)^{2^{16}+1}\frac{2^{n-m}}{2^n}>\left(\frac{1}{2}\right)^{2^{16}+17}\frac{q(q-1)}{2^n}.\qedhere\end{align*}
\end{proof}
We now address the regime $2^{\frac{n-m}{2}+8}<q\leq 2^{\frac{n+m}{2}-3}$, in which
\begin{equation*}
\min\left\{\frac{q(q-1)}{2^n},\,\frac{q}{2^{\frac{n+m}{2}}},\,1\right\}=\Theta\left(\frac{q}{2^{\frac{n+m}{2}}}\right).
\end{equation*}
\begin{proposition}\label{prop:main}
If $2^{\frac{n-m}{2}+8}<q\leq 2^{\frac{n+m}{2}-3}$, then
\begin{equation*}
{\bf Adv}_{n,m}(q)=\Omega\left(\frac{q}{2^{\frac{n+m}{2}}}\right).
\end{equation*}
\end{proposition}
For $\omega=(\omega_i)_{i=1}^q\in{({\mathcal B}_{n-m})^q}$, let 
\begin{equation*}
\text{Col}(\omega):=\#\{1\leq i<j\leq q\mid\omega_i=\omega_j\}
\end{equation*}
be the number of collisions in the sequence $\omega$ and let 
\begin{equation*}
X(\omega):=\text{Col}(\omega)-{\mathbb E}\,\text{Col}=\sum_{\alpha\in{\mathcal B}_{n-m}}\binom{d_\alpha(\omega)}{2}-\binom{q}{2}\frac{1}{2^{n-m}},
\end{equation*}
where all the probabilistic notions, such as expectation, here and below, are with respect to the uniform distribution on ${({\mathcal B}_{n-m})^q}$.
Proposition \ref{prop:main} will easily follow from the following technical lemmas.
\begin{lemma}\label{lem:lnRupper}
Suppose that $q$ is a power of $2$. Then,
\begin{equation*}R\leq \exp\left(\frac{q^2}{2^{n+m+1}}-\frac{1}{2^m}X\right).\end{equation*}
\end{lemma}
The proof of Lemma \ref{lem:lnRupper} will be given in Section \ref{sec:lnRupper}.
\begin{lemma}\label{lem:polynom}
If $q>2^{\frac{n-m}{2}+8}$, then
\begin{equation*}
\Pr\left(X>\frac{q}{10\cdot 2^{\frac{n-m}{2}}}\right)>\frac{1}{400}.
\end{equation*}
\end{lemma}
The proof of Lemma \ref{lem:polynom} will be given in Section \ref{sec:polynom}.
We proceed to prove Proposition \ref{prop:main}.
\begin{proof} [Proof of Proposition \ref{prop:main}]
With no loss of generality we may assume that $q$ is a power of $2$. 
Let 
\begin{equation*}
S:=\{\omega\in({\mathcal B}_{n-m})^q\mid X(\omega)>\frac{q}{10\cdot 2^{\frac{n-m}{2}}}\}.
\end{equation*} 
By Lemma \ref{lem:polynom}, $\frac{|S|}{2^{(n-m)q}}=\Pr(S)>\frac{1}{400}$.
For every $\omega\in S$,
\begin{equation*}
\frac{q^2}{2^{n+m+1}}-\frac{1}{2^m}X(\omega)<\frac{q^2}{2^{n+m+1}}-\frac{1}{2^m}\cdot\frac{q}{10\cdot 2^{\frac{n-m}{2}}}
=-\frac{1}{10}\left(1-\frac{5q}{2^{\frac{n+m}{2}}}\right)\frac{q}{2^{\frac{n+m}{2}}}<- \frac{3}{80}\cdot\frac{q}{2^{\frac{n+m}{2}}}
\end{equation*}
and hence, by Lemma \ref{lem:lnRupper},
\begin{equation*}
1-R(\omega)> 1-\exp\left(- \frac{3}{80}\cdot\frac{q}{2^{\frac{n+m}{2}}}\right).
\end{equation*}
Therefore, by \eqref{eq:AdvR},
\begin{align*}{\bf Adv}_{n,m}(q)&\geq\frac{1}{2^{(n-m)q}}\left|\sum_{\omega\in S}\left(R(\omega)-1\right)\right|\geq\frac{|S|}{2^{(n-m)q}}\left(1-\exp\left(- \frac{3}{80}\cdot\frac{q}{2^{\frac{n+m}{2}}}\right)\right)\\
&>\frac{1}{400}\left(1-\exp\left(- \frac{3}{80}\cdot\frac{q}{2^{\frac{n+m}{2}}}\right)\right)=\Omega\left(\frac{q}{2^{\frac{n+m}{2}}}\right).
\qedhere\end{align*}
\end{proof}
Now we can prove Theorem \ref{thm:main}.

\begin{proof}[Proof of Theorem \ref{thm:main}]
The upper bound was already demonstrated in the introduction, so we only need to show that
\begin{equation*}
{\bf Adv}_{n,m}(q)=\Omega\left(\min\left\{\frac{q(q-1)}{2^n},\,\frac{q}{2^{\frac{n+m}{2}}},\,1\right\}\right).
\end{equation*}
If $q\leq 2^{\frac{n-m}{2}+8}$, then by Proposition \ref{prop:small},
\begin{equation*}
{\bf Adv}_{n,m}(q)=\Omega\left(\frac{q(q-1)}{2^n}\right).
\end{equation*}
If $2^{\frac{n-m}{2}+8}<q\leq 2^{\frac{n+m}{2}-3}$, then by Proposition \ref{prop:main},
\begin{equation*}
{\bf Adv}_{n,m}(q)=\Omega\left(\frac{q}{2^{\frac{n+m}{2}}}\right).
\end{equation*}
Finally, if $q>2^{\frac{n+m}{2}-3}$, then by Proposition \ref{prop:main},
\begin{equation*}{\bf Adv}_{n,m}(q)\geq {\bf Adv}_{n,m}\left(2^{\frac{n+m}{2}-3}\right)=\Omega\left(\frac{2^{\frac{n+m}{2}-3}}{2^{\frac{n+m}{2}}}\right)=\Omega(1).\qedhere\end{equation*}
\end{proof}

\section{Proof of Lemma \ref{lem:lnRupper}}\label{sec:lnRupper}
For every positive real $t$ and nonnegative integer $k$, denote
$L_t(k):=\ln \prdct {k}{t}+\binom{k}{2}\frac{1}{t}$.

\begin{lemma}
For every positive real $t$ and positive integer $k\leq t/2$, it holds that
\begin{subequations} 
\begin{align}\label{eq:Wlower}
L_t(k)&\geq -\frac{k^3}{3t^2},\\
\label{eq:moreW}
\frac{1}{k}L_t(k)-\frac{1}{2k}L_{2t}(2k)&\leq\frac{k}{2t^2}-\frac{2k}{2(2t)^2},
\end{align}
and consequently, for every positive integer $\ell$,
\begin{equation}\label{eq:moreWtelescopic}
\frac{1}{k}L_t(k)-\frac{1}{2^{\ell}k}L_{2^{\ell}t}(2^{\ell}k)\leq\frac{k}{2t^2}-\frac{2^{\ell}k}{2(2^{\ell}t)^2}.
\end{equation}
\end{subequations} 
\end{lemma}
\begin{proof}
For every $x<1$, let $\varphi(x):=x+x^2+\ln(1-x)$. Then, for every $x<1$,
\begin{equation*}
\varphi'(x)=1+2x-\frac{1}{1-x}=\frac{x(1-2x)}{1-x}.
\end{equation*}
Therefore, $\varphi$ is increasing in the interval $[0,\frac{1}{2}]$. 
In particular, for every $0\leq x\leq 1/2$,
\begin{equation}\label{eq:ln_upper}\ln(1-x)+x=-x^2+\varphi(x)\geq -x^2+\varphi(0)=-x^2. \end{equation}
The estimate \eqref{eq:Wlower} immediately follows:
\begin{equation*}
L_t(k)=\sum_{j=0}^{k-1}\left(\ln\left(1-\frac{j}{t}\right)+\frac{j}{t}\right)\geq -\sum_{j=0}^{k-1}\frac{j^2}{t^2}=-\frac{k(k-1)(2k-1)}{6t^2}\geq-\frac{k^3}{3t^2}.
\end{equation*}
 To get \eqref{eq:moreW}, observe first that
\begin{equation*}
\left(\frac{\prod_{j=0}^{k-1}\left(1-\frac{2j+1}{2t}\right)}{\prod_{j=0}^{k-1}\left(1-\frac{2j}{2t}\right)}\right)^2=\frac{\prod_{j=0}^{k-1}\left(1-\frac{2j+1}{2t}\right)^2}{\prod_{j=0}^{k-1}\left(1-\frac{2j}{2t}\right)\left(1-\frac{2j+2}{2t}\right)}\left(1-\frac{k}{t}\right)\geq 1-\frac{k}{t},
\end{equation*}
and hence
\begin{equation*}
\frac{{\prdct {2k}{2t}}}{({\prdct k t})^2}
=\frac{\prod_{j=0}^{k-1}\left(1-\frac{2j}{2t}\right)\left(1-\frac{2j+1}{2t}\right)}{\prod_{j=0}^{k-1}\left(1-\frac{2j}{2t}\right)^2}=\frac{\prod_{j=0}^{k-1}\left(1-\frac{2j+1}{2t}\right)}{\prod_{j=0}^{k-1}\left(1-\frac{2j}{2t}\right)}\geq\sqrt{1-\frac{k}{t}}.
\end{equation*}
Therefore, 
\begin{align*}
\frac{1}{2k}L_{2t}(2k)-\frac{1}{k}L_t(k)&=\frac{1}{2k}\ln\frac{{\prdct {2k}{2t}}}{({\prdct k t})^2}+\frac{1}{4t}\geq \frac{1}{2k}\ln\sqrt{1-\frac{k}{t}}+\frac{1}{4t}\\
&=\frac{1}{4k}\left(\ln\left(1-\frac{k}{t}\right)+\frac{k}{t}\right)\geq -\frac{1}{4k}\left(\frac{k}{t}\right)^2=\frac{2k}{2(2t)^2}-\frac{k}{2t^2},
\end{align*}
where the second inequality holds by \eqref{eq:ln_upper}, and \eqref{eq:moreW} follows.

Finally, for every $1\leq j\leq \ell$, by applying \eqref{eq:moreW} to $2^{j-1} k$ and $2^{j-1} t$, it holds that
\begin{equation*}
\frac{1}{2^{j-1} k}L_{2^{j-1} t}(2^{j-1} k)-\frac{1}{2^j k}L_{2^j t}(2^j k)\leq\frac{2^{j-1} k}{2(2^{j-1} t)^2}-\frac{2^j k}{2(2^j t)^2},
\end{equation*}
and \eqref{eq:moreWtelescopic} follows by summing up these inequalities and collapsing the obtained telescopic sums, as follows:
\begin{align*}
\frac{1}{k}L_t(k)-\frac{1}{2^{\ell}k}L_{2^{\ell}t}(2^{\ell}k)&=\sum_{j=1}^{\ell}\left(\frac{1}{2^{j-1} k}L_{2^{j-1} t}(2^{j-1} k)-\frac{1}{2^j k}L_{2^j t}(2^j k)\right)\\
&\leq\sum_{j=1}^{\ell}\left(\frac{2^{j-1} k}{2(2^{j-1} t)^2}-\frac{2^j k}{2(2^j t)^2}\right)=\frac{k}{2t^2}-\frac{2^{\ell}k}{2(2^{\ell}t)^2}.
\qedhere\end{align*}
\end{proof}
Let ${\mathcal D}$ be the set of sequences $(d_{\alpha})_{\alpha\in{\mathcal B}_{n-m}}$ of nonnegative integers such that $d_{\alpha}\leq 2^m$ for every $\alpha\in{\mathcal B}_{n-m}$ and $\sum_{\alpha\in{\mathcal B}_{n-m}}d_{\alpha}=q$.
\begin{lemma}\label{lem:F}
Suppose that $q$ is a power of $2$. 
For every $(d_{\alpha})_{\alpha\in{\mathcal B}_{n-m}}\in{\mathcal D}$,
\begin{equation*}
\sum_{\alpha\in{\mathcal B}_{n-m}}L_{2^m}(d_{\alpha})\leq \begin{cases}
0 &\quad q<2^{n-m}, \\
2^{n-m}L_{2^m}\left(\frac{q}{2^{n-m}}\right) &\quad q\geq 2^{n-m}.
\end{cases}
\end{equation*}
\end{lemma}
\begin{proof}
Note that $L_{2^m}(0)=L_{2^m}(1)=0$. For every integer $0\leq d\leq 2^m-1$, 
\begin{equation*}
L_{2^m}(d+1)-L_{2^m}(d)=\ln \frac{\prdct {d+1}{2^m}}{\prdct d {2^m}}+\left(\binom{d+1}{2}-\binom{d}{2}\right)\frac{1}{2^m}=\ln\left(1-\frac{d}{2^m}\right)+\frac{d}{2^m}.
\end{equation*}
Hence, since the function $x\mapsto\ln(1-x)+x$ is strictly decreasing in the interval $[0,1)$, it holds that for every $0\leq d_1<d_2\leq 2^m-1$,
\begin{equation*}
L_{2^m}(d_2+1)-L_{2^m}(d_2)<L_{2^m}(d_1+1)-L_{2^m}(d_1),
\end{equation*}
i.e.,
\begin{equation*}
L_{2^m}(d_1)+L_{2^m}(d_2+1)<L_{2^m}(d_1+1)+L_{2^m}(d_2).
\end{equation*}
It follows that the maximum $\sum_{\alpha\in{\mathcal B}_{n-m}}L_{2^m}(d_{\alpha})$ for $(d_{\alpha})_{\alpha\in{\mathcal B}_{n-m}}\in{\mathcal D}$ is attained for sequences $(d_{\alpha})_{\alpha\in{\mathcal B}_{n-m}}$ for which $|d_{\alpha_1}-d_{\alpha_2}|\leq 1$ for every $\alpha_1,\alpha_2\in{\mathcal B}_{n-m}$.
In particular, if $q\geq 2^{n-m}$ then the maximum of $\sum_{\alpha\in{\mathcal B}_{n-m}}L_{2^m}(d_{\alpha})$ for $(d_{\alpha})_{\alpha\in{\mathcal B}_{n-m}}\in{\mathcal D}$ is attained at the sequence $(d_{\alpha})_{\alpha\in{\mathcal B}_{n-m}}$ such that $d_{\alpha}=q/2^{n-m}$ for every $\alpha\in{\mathcal B}_{n-m}$; 
if $q<2^{n-m}$ then the maximum of $\sum_{\alpha\in{\mathcal B}_{n-m}}L_{2^m}(d_{\alpha})$ for $(d_{\alpha})_{\alpha\in{\mathcal B}_{n-m}}\in{\mathcal D}$ is attained at any $(d_{\alpha})_{\alpha\in{\mathcal B}_{n-m}}\in{\mathcal D}$ for which $d_{\alpha}\leq 1$ for every $\alpha\in{\mathcal B}_{n-m}$.
The lemma follows.
\end{proof}

\begin{proof}[Proof of Lemma \ref{lem:lnRupper}]
Let $\omega=(\omega_i)_{i=1}^q\in{({\mathcal B}_{n-m})^q}$. If $d_{\alpha}(\omega)>2^m$ for some $\alpha\in{\mathcal B}_{n-m}$, then surely
\begin{equation*}R(\omega)=0<\exp\left(\frac{q^2}{2^{n+m+1}}-\frac{1}{2^m}X(\omega)\right).\end{equation*}
We therefore assume that $d_{\alpha}(\omega)\leq 2^m$ for every $\alpha\in{\mathcal B}_{n-m}$, and hence $\left(d_{\alpha}(\omega)\right)_{\alpha\in{\mathcal B}_{n-m}}\in{\mathcal D}$.
Note that
\begin{equation*}
\ln R(\omega)+\frac{1}{2^m}X(\omega)=\left(\sum_{\alpha\in{\mathcal B}_{n-m}}L_{2^m}(d_{\alpha}(\omega))\right)-L_{2^n}(q).
\end{equation*}
Hence, if $q<2^{n-m}$ then by Lemma \ref{lem:F} and \eqref{eq:Wlower},
\begin{equation*}\ln R(\omega)+\frac{1}{2^m}X(\omega)\leq -L_{2^n}(q)\leq \frac{q^3}{3\cdot 2^{2n}}
<\frac{q^2}{2^{n+m+1}},
\end{equation*}
and if $q\geq 2^{n-m}$ then by Lemma \ref{lem:F} and \eqref{eq:moreWtelescopic},
\begin{align*}
\ln R(\omega)+\frac{1}{2^m}X(\omega)&\leq 2^{n-m}L_{2^m}\left(\frac{q}{2^{n-m}}\right)-L_{2^n}(q)=q\left(\frac{1}{\frac{q}{2^{n-m}}}L_{2^m}\left(\frac{q}{2^{n-m}}\right)-\frac{1}{q}L_{2^n}(q)\right)\\
&\leq q\left(\frac{\frac{q}{2^{n-m}}}{2(2^{m})^2}-\frac{q}{2(2^{n})^2}\right)<q\frac{\frac{q}{2^{n-m}}}{2(2^{m})^2}=\frac{q^2}{2^{n+m+1}},
\end{align*}
and the result follows.
\end{proof}

\section{Proof of Lemma \ref{lem:polynom}}\label{sec:polynom} 

Denote $p:=\frac{1}{2^{n-m}}$ and let $\tilde{X}:=\frac{1}{q\sqrt{p}}X$. 
The proof of Lemma \ref{lem:polynom} will be based on the following technical claim.

\begin{claim}\label{claim:real_polynom}
If $q>2^{\frac{n-m}{2}+8}$, then there is a real polynomial $\varphi$ satisfying the following properties.
\begin{subequations} 
\begin{align}\label{eq:1}
&\varphi(x)\leq 0 \text{ for every } x\leq \frac{1}{10},\\
\label{eq:2}
&\varphi(x)<200 \text{ for every real } x,\\
\label{eq:3}
&\mathbb{E}\,\varphi(\tilde{X})>\frac{1}{2}.
\end{align}
\end{subequations} 
\end{claim}

We will first show how Lemma \ref{lem:polynom} may be deduced from Claim \ref{claim:real_polynom}. 

\begin{proof}[Proof of Lemma \ref{lem:polynom}]
Let $\varphi$ be as in Claim \ref{claim:real_polynom}.
By \eqref{eq:2}, the random variable $200-\varphi(\tilde{X})$ is nonnegative.
Hence, by Markov's inequality,
\begin{equation}\label{eq:markov}
\Pr\left(\varphi(\tilde{X})\leq 0\right)=\Pr\left(200-\varphi(\tilde{X})\geq 200\right)\leq \frac{\mathbb{E}\left(200-\varphi(\tilde{X})\right)}{200}=1-\frac{\mathbb{E}\,\varphi(\tilde{X})}{200}.
\end{equation}
By \eqref{eq:1}, 
$\left\{X\leq\frac{q\sqrt{p}}{10}\right\}=\left\{\tilde{X}\leq\frac{1}{10}\right\}\subseteq\{\varphi(\tilde{X})\leq 0\}$.
Therefore, by using \eqref{eq:markov} and \eqref{eq:3},
\begin{equation*}
\Pr\left(X>\frac{q}{10\cdot 2^{\frac{n-m}{2}}}\right)=\Pr\left(X>\frac{q\sqrt{p}}{10}\right)\geq\Pr\left(\varphi(\tilde{X})>0\right)\geq\frac{\mathbb{E}\,\varphi(\tilde{X})}{200}>\frac{1}{400}.
\qedhere\end{equation*}
\end{proof}

We proceed to prove Claim \ref{claim:real_polynom}.
A straightforward calculation (which we include in the appendix, for completeness) yields that
\begin{subequations} \label{eq:moments}
\begin{align}
\mathbb{E}X=&0,\label{eq:E1}\\
\mathbb{E}X^2=&\binom{q}{2}p\left(1-p\right),\label{eq:E2}\\
\mathbb{E}X^3=&6\binom{q}{3}p^2\left(1-p\right)+\binom{q}{2}p\left(1-p\right)\left(1-2p\right),\label{eq:E3}\\
\mathbb{E}X^4=&18\binom{q}{4}p^2\left(1-p\right)\left(1+3p\right)+18\binom{q}{3}p^2\left(1-p\right)\left(3-5p\right)\label{eq:E4}\\
&+\binom{q}{2}p\left(1-p\right)\left(1-3p+3p^2\right).\nonumber
\end{align}
\end{subequations}

\begin{proof}[Proof of Claim \ref{claim:real_polynom}]
For every real $x$, let
\begin{equation*}
\varphi(x):=-\left(x+\frac{5}{2}\right)^2\left(x-\frac{1}{10}\right)\left(x-5\right)=-x^4+\frac{1}{10}x^3+\frac{75}{4}x^2+\frac{235}{8}x-\frac{25}{8}.\end{equation*}
Clearly, $\varphi(x)\leq 0$ for every $x\leq \frac{1}{10}$.
For every real $x$,
\begin{equation*}
\varphi'(x)=-4\left(x+\frac{5}{2}\right)\left(x-\frac{103-\sqrt{29409}}{80}\right)\left(x-\frac{103+\sqrt{29409}}{80}\right).
\end{equation*}
It follows that $\varphi(x)\leq\varphi\left(\frac{103+\sqrt{29409}}{80}\right)<200$ for every real $x$.
It remains to show that ${\mathbb E}\,\varphi(\tilde{X})>\frac{1}{2}$.

First, note that
$(1-p)(1+3p)\leq\frac{21}{16}$
(this may be verified by direct computation for $n-m=1$, and if $n-m\geq 2$ then $p\leq\frac{1}{4}$ and hence $(1-p)(1+3p)\leq\left(1-\frac{1}{4}\right)\left(1+\frac{3}{4}\right)=\frac{21}{16}$, since the function $x\mapsto (1-x)(1+3x)$ is increasing in the interval $[0,\frac{1}{3}]$).
 Therefore,
\begin{equation*}
18\binom{q}{4}p^2(1-p)(1+3p)<18\cdot\frac{q^4}{24}\cdot p^2\cdot\frac{21}{16}=\left(1-\frac{1}{2^6}\right)(q\sqrt{p})^4.
\end{equation*}
Next, note that $(1-p)(3-5p)<\frac{3}{4\sqrt{p}}$ 
(this may be verified by direct computation for $1\leq n-m\leq 3$, and if $n-m\geq 4$ then $(1-p)(3-5p)<3\leq\frac{3}{4\sqrt{p}}$).
Therefore, since $q\sqrt{p}>2^8$,
\begin{equation*}
18\binom{q}{3}p^2(1-p)(3-5p)<3q^3p^2\frac{3}{4\sqrt{p}}=\frac{9}{2^2}(q\sqrt{p})^3<\frac{9}{2^{10}}(q\sqrt{p})^4.
\end{equation*}
Additionally, 
\begin{equation*}
\binom{q}{2}p^2(1-p)(1-3p+3p^2)<\frac{1}{2}q^2p=\frac{1}{2}(q\sqrt{p})^2<\frac{1}{2^{17}}(q\sqrt{p})^4.
\end{equation*}
Therefore, by \eqref{eq:E4}, 
\begin{equation*}
\mathbb{E}X^4<\left(1-\frac{1}{2^6}+\frac{9}{2^{10}}+\frac{1}{2^{17}}\right)(q\sqrt{p})^4<(q\sqrt{p})^4,
\end{equation*}
i.e., ${\mathbb E}\tilde{X}^4<1$. 
Additionally,
${\mathbb E}\tilde{X}=0$ by \eqref{eq:E1},
${\mathbb E}\tilde{X}^2=\left(1-\frac{1}{q}\right)\frac{1-p}{2}>\left(1-\frac{1}{2^8}\right)\frac{1}{4}$ by \eqref{eq:E2}, 
and $\mathbb{E}\tilde{X}^3\geq 0$ by \eqref{eq:E3}.
Therefore,
\begin{equation*}
{\mathbb E}\,\varphi(\tilde{X})=-\mathbb{E}\tilde{X}^4+\frac{1}{10}\mathbb{E}\tilde{X}^3+\frac{75}{4}\mathbb{E}\tilde{X}^2+\frac{235}{8}\mathbb{E}\tilde{X}-\frac{25}{8}\geq-1+\frac{75}{4}\left(1-\frac{1}{2^8}\right)\frac{1}{4}-\frac{25}{8}>\frac{1}{2}.
\qedhere\end{equation*}
\end{proof}

\subsection*{Acknowledgments}
We thank Ron Peled for fruitful discussions.

\subsection*{Funding}
This research was partially supported by the Bar-Ilan University Center for Research in Applied Cryptography and Cyber Security, and the Center for Cyber Law and Policy at the University of Haifa, both in conjunction with the Israel National Cyber Bureau in the Prime Minister's Office;
the Israel Science Foundation (ISF, grant number 3380/19); 
and a joint funding research grant of the U.S. National Science Foundation and the U.S.--Israel Binational Science Foundation (NSF--BSF, grant number 2018640).

\appendix

\section{Proof of \eqref{eq:moments}}
Let ${\mathcal E}:=\{\{i,j\}\mid 1\leq i<j\leq q\}$.
For every $(i,j)\in{\mathcal E}$, let $Y_{\{ i,j\}}$ be the indicator function of the event $\{\omega_i=\omega_j\}$, and let $X_{\{ i,j\}}:=Y_{\{ i,j\}}-{\mathbb E}Y_{\{ i,j\}}=Y_{\{ i,j\}}-p$. Evidently ${\mathbb E}X_e=0$ for every $e\in{\mathcal E}$ and \eqref{eq:E1} follows, since $X=\sum_{e\in{\mathcal E}}X_e$.
Since the events $(\{\omega_i=\omega_j\})_{(i,j)\in{\mathcal E}}$ are mutually independent, it holds that for every $e_1,e_2\in{\mathcal E}$,
\begin{equation*}\label{eq:cov}
{\mathbb E}X_{e_1}X_{e_2}={\rm Cov}(X_{e_1},X_{e_2})={\rm Cov}(Y_{e_1},Y_{e_2})=
\begin{cases}
0 & \quad e_1\neq e_2,\\
p(1-p) & \quad e_1=e_2,
\end{cases}
\end{equation*} 
and \eqref{eq:E2} follows.
For every $e_1,e_2,e_3\in{\mathcal E}$, 
\begin{equation*}
\mathbb{E}Y_{e_1}Y_{e_2}Y_{e_3}=\begin{cases}
p &\quad e_1=e_2=e_3, \\
p^2 &\quad |\{e_1,e_2,e_3\}|=2 \text{ or } |e_1\cup e_2\cup e_3|=3,\\
p^3 &\quad \text{otherwise,} 
\end{cases}\end{equation*}
and on the other hand,
\begin{align*}\mathbb{E}Y_{e_1}Y_{e_2}Y_{e_3}=& \mathbb{E}X_{e_1}X_{e_2}X_{e_3}+p\sum_{1\leq i_1<i_2\leq 3}\mathbb{E}X_{e_{i_1}}X_{e_{i_2}}+p^2\sum_{i=1}^3\mathbb{E}X_{e_i}+p^3\\
=&\mathbb{E}X_{e_1}X_{e_2}X_{e_3}+p^3+
\begin{cases}
3p^2(1-p) &\quad e_1=e_2=e_3, \\
p^2(1-p) &\quad |\{e_1,e_2,e_3\}|=2,\\
0 &\quad \text{otherwise.} 
\end{cases}\end{align*} 
Hence, for every $e_1,e_2,e_3\in{\mathcal E}$, 
\begin{equation}\label{eq:third}
{\mathbb E}X_{e_1}X_{e_2}X_{e_3}=\begin{cases}
p(1-p)(1-2p) &\quad e_1=e_2=e_3, \\
p^2(1-p) &\quad|\{e_1,e_2,e_3\}|=3 \text{ and } |e_1\cup e_2\cup e_3|=3,\\
0 &\quad \text{otherwise,} 
\end{cases}
\end{equation}
and \eqref{eq:E3} follows.
We proceed to prove \eqref{eq:E4}.
Let
\begin{align*}
{\mathcal P}&:=\{(e_1,e_2,e_3,e_4)\in{\mathcal E}^4\mid \forall 1\leq i\leq 4: |\{1\leq j\leq 4\mid e_j=e_i\}|=2\},\\
{\mathcal T}&:=\{(e_1,e_2,e_3,e_4)\in{\mathcal E}^4\mid \text{the graph that the edges } e_1,e_2,e_3,e_4 \text{ form contains a triangle}\},\\
{\mathcal Q}&:=\{(e_1,e_2,e_3,e_4)\in{\mathcal E}^4\mid \text{the graph that the edges } e_1,e_2,e_3,e_4 \text{ form a quadraliteral}\}.
\end{align*}
For every $e_1,e_2,e_3,e_4\in{\mathcal E}$, 
\begin{equation*}\label{eq:4}\mathbb{E}Y_{e_1}Y_{e_2}Y_{e_3}Y_{e_4}=\begin{cases}
p &\quad e_1=e_2=e_3=e_4, \\
p^2 &\quad |\{e_1,e_2,e_3,e_4\}|=2 \text{ or } |e_1\cup e_2\cup e_3\cup e_4|=3, \\
p^3 &\quad \text{either } |\{e_1,e_2,e_3,e_4\}|=3 \text{ or } \\
 & \quad (e_1,e_2,e_3,e_4)\in{\mathcal T}\cup{\mathcal Q} \text{ (but not both),}\\
p^4 &\quad \text{otherwise,} 
\end{cases}\end{equation*}
and on the other hand, by using \eqref{eq:third}, 
\begin{align*}\mathbb{E}&Y_{e_1}Y_{e_2}Y_{e_3}Y_{e_4}=\mathbb{E}X_{e_1}X_{e_2}X_{e_3}X_{e_4}+p\sum_{1\leq i_1<i_2<i_3\leq 4}\mathbb{E}X_{e_{i_1}}X_{e_{i_2}}X_{e_{i_3}}\\
&+p^2\sum_{1\leq i_1<i_2\leq 4}\mathbb{E}X_{e_{i_1}}X_{e_{i_2}}+p^3\sum_{i=1}^4\mathbb{E}X_{e_i}+p^4
=\mathbb{E}X_{e_1}X_{e_2}X_{e_3}X_{e_4}+p^4\\
&+\begin{cases}
4p^2(1-p)(1-2p)+6p^3(1-p) &\quad e_1=e_2=e_3=e_4, \\
p^2(1-p)(1-2p)+3p^3(1-p) & \quad |\{e_1,e_2,e_3,e_4\}|=2 \text{ but } (e_1,e_2,e_3,e_4)\notin{\mathcal P},\\
2p^3(1-p) & \quad (e_1,e_2,e_3,e_4)\in{\mathcal P},\\
3p^3(1-p) & \quad |e_1\cup e_2\cup e_3\cup e_4|=3 \text{ and } (e_1,e_2,e_3,e_4)\in{\mathcal T},\\
p^3(1-p) & \quad \text{either } |\{e_1,e_2,e_3,e_4\}|=3 \text{ or }\\
& \quad (e_1,e_2,e_3,e_4)\in{\mathcal T}\text{ (but not both),}\\
0 &\quad \text{otherwise.} 
\end{cases}\end{align*}
Hence, for every $e_1,e_2,e_3,e_4\in{\mathcal E}$,
\begin{align*}\mathbb{E}&X_{e_1}X_{e_2}X_{e_3}X_{e_4}\\
&=
\begin{cases}
p(1-p)(1-3p+3p^2) &\quad e_1=e_2=e_3=e_4, \\
p^2(1-p)(1-2p) &\quad |\{e_1,e_2,e_3,e_4\}|=3 \text{ and } |e_1\cup e_2\cup e_3\cup e_4|=3, \\
p^2(1-p)^2 &\quad (e_1,e_2,e_3,e_4)\in{\mathcal P},\\
p^3(1-p) &\quad (e_1,e_2,e_3,e_4)\in{\mathcal Q}, \\
0 &\quad \text{otherwise,} 
\end{cases}
\end{align*}
and \eqref{eq:E4} follows.
\end{document}